\documentclass[12pt]{article}
\usepackage[dvips]{graphicx}
\usepackage{t1enc,amssymb,amsbsy,amsthm,amsmath,amsfonts}

\def\e{\hbox {\rm e}}

\renewcommand{\Re}{{\rm I\kern-0.16em R}}

\def\@begintheorem#1#2{\trivlist \item[\hskip \labelsep{\bf #1\ #2}]}
\def\@opargbegintheorem#1#2#3{\trivlist
      \item[\hskip \labelsep{\bf #1\ #2\ (#3)}]}

\newtheorem{proposition}{Proposition}[section] 

\newtheorem{thm}[proposition]{Theorem}

\newtheorem{corollary}[proposition]{Corollary}
\newtheorem{example}[proposition]{Example}
\newtheorem{remark}[proposition]{Remark}

\def\e{\hbox {\rm e}}
\def\P{{\bf P}}
\def\R{{\bf R}}
\def\R{{\bf R}}

\def\E{{\bf E}}

\def\cM{{\cal M}}

\numberwithin{equation}{section}

\begin{document}

\author{
Paavo Salminen
\\{\small \AA bo Akademi,}
\\{\small Mathematical Department,}
\\{\small F\"anriksgatan 3 B,}
\\{\small FIN-20500 \AA bo, Finland,} 
\\{\small email: phsalmin@abo.fi}
}
\title{Optimal stopping, Appell polynomials and Wiener-Hopf factorization}
\date{}
\maketitle
\begin{abstract}
In this paper we study the optimal stopping problem for L\'evy processes studied by Novikov and Shiryayev 
in \cite{novikovshiryaev04}. In particular, we are interested in finding the representing measure of the value function. It is seen that that this can be expressed in terms of the Appell polynomials. An important tool in our approach and computations is the Wiener-Hopf factorization.
\end{abstract} 



\newpage
\section{Introduction}
\label{sec0}
Let $X=\{X_t\,:\, t\geq 0\}$ denote a real-valued L\'evy process and $ \P_x $ the probability measure associated with $X$ when initiated from $ x $. Furthermore, $ \lbrace{\cal F}_t\rbrace $
denotes the natural filtration generated by $X$ and  ${\cM}$ the set of all 
stopping times $\tau$ with
respect to $ \lbrace{\cal F}_t\rbrace. $

We are interested in the following optimal stopping problem:
\\
\\
 {\sl Find a function $V$ and a stopping time $\tau^*$ such that
\begin{equation}
\label{osp}
V(x)=\sup_{\tau\in{\cM}}\E_x({\rm e}^{-r\tau}g(X_\tau))=\E_x({\rm e}^{-r\tau^*}g(X_\tau^*)),
\end{equation}
where $g(x):=(x^+)^n,\, n=1,2,\dots,$ and $ r\geq 0. $}
\\
\\
It is striking that the solution of this problem  can be characterized fairly explicitly for a general L\'evy process whose L\'evy measure satisfies some integrability conditions. 
This solution is essentially due to 
Novikov and Shiryaev
\cite{novikovshiryaev04} who found it for random walks. Construction was lifted to the framework of L\'evy processes by Kyprianou and Surya \cite{kyprianousurya05} (see also Kyprianou \cite{kyprianou06} Section 9). In Novikov and Shiryayev \cite{novikovshiryayev07} the corresponding stopping problem with arbitrary power $\gamma>0,$ i.e., $g(x):=(x^+)^\gamma$ is analyzed. For related problems for random walks and L\'evy processes, see Darling et al. \cite{darlingliggettaylor72}, and Mordecki \cite{mordecki02}.

The so called Appell polynomials (in  \cite{novikovshiryayev07} a more general concept of Appell function is introduced) play a central role in the development directed by Novikov and Shiryayev. 
In the solution, the function $V$ is described as the expectation of a function of the maximum of the L\'evy process up to an independent exponential time $T.$ In the paper by Salminen and Mordecki  \cite{mordeckisalminen07} the appearance of a function of maximum is explained via the Wiener-Hopf-Rogozin factorization of L\'evy processes and the Riesz decomposition and representation  of excessive functions. 

The aim of this note is to study the Novikov-Shiryayev solution in light of the results in 
\cite{mordeckisalminen07}. In particular, we focus on finding the representing measure of the excessive function $V.$ In the spectrally positive case the representing measure has a clean expression in terms of the Appell polynomials  of $X_T.$ In the general case we are able to find the Laplace transform of the representing measure but unable to find a "nice" inversion. If $X$ is spectrally negative, this Laplace transform is expressed in terms of the L\'evy-Khintchine exponent of $X.$

In the next section basic properties and examples of Appell polynomials are discussed. Therein is a short subsection on L\'evy processes where we present notation, assumptions and features central and important for the application in optimal stopping. In Section 3 -  to make the paper more readable - the results from \cite{novikovshiryaev04} and \cite{mordeckisalminen07} are shortly recalled. After this we proceed with the main results of the paper concerning the  representing measure of $V.$ 

\section{Appell polynomials}
\label{sec1}

\subsection{Basic properties}
\label{sec11}

Let $\eta$ be a random variable with some exponential moments, i.e., there exists $u>0$ such that
$$
\E\left(\e^{u\vert \eta\vert}\right)<\infty.
$$
A family of polynomials $\{Q^{(\eta)}_k\,;\,k=0,1,2,\dots\}$, the Appell polynomials associated with $\eta,$ are defined via
\begin{equation}
\label{appell}
\frac{\e^{ux}}{\E\left(\e^{u\eta}\right)}=\sum_{k=0}^\infty \frac{u^k}{k!} Q^{(\eta)}_k(x).
\end{equation} 
Putting here $x=\eta +z$ and taking expectations we obtain easily the so called mean value property of 
the Appell polynomials
\begin{equation}
\label{appell1}
\E\left( Q^{(\eta)}_k(\eta+z)\right)=z^k.
\end{equation} 
Writing
\[ 
\frac{\e^{ux}}{\E\left(\e^{u\eta}\right)}= 
\sum_{i=0}^\infty \frac{u^i}{i!} x^i
\sum_{k=0}^\infty \frac{u^k}{k!} Q^{(\eta)}_k(0)
\]
and rearranging yields
\begin{equation}
\label{appell11}
 Q^{(\eta)}_m(x)=\sum_{k=0}^m{m\choose k}\,x^k\,Q^{(\eta)}_{m-k}(0).
\end{equation} 
Consequently, taking derivatives, it is seen that the Appell polynomials satisfy  
 \begin{eqnarray}
\label{appellx}
&&
\frac{d}{dx}Q^{(\eta)}_m(x)=m\,Q^{(\eta)}_{m-1}(x),
\\
&&
 Q^{(\eta)}_0(x)=1 \ {\rm for\ all\ } x.
\end{eqnarray}
The recursion in (\ref{appellx}) if combined with the normalization (cf. (\ref{appell1}))
$$ 
\E\left(Q_n(\eta)\right)=0 \ n=1,2,\dots
$$
provides an alternative definition of the Appell polynomials. 

Recall that the cumulant function associated with $\eta$ is defined as
 \begin{equation}
\label{appell3}
K(u):=\log\E\left(\e^{u\eta}\right)=\sum_{k=0}^\infty \kappa_k\,\frac{u^k}{k!},
 \end{equation}
and the coefficients $\kappa_k,\, k=0,1,\dots,$ in the above McLaurin expansion are called the cumulants of $\eta.$ We remark that 
$$ 
\kappa_1=\E(\eta),\qquad \kappa_2= \E\left(\left(\eta-\E(\eta)\right)^2\right)
,\qquad \kappa_3= \E\left(\left(\eta-\E(\eta)\right)^3\right)
$$
but
$$ 
\kappa_4\not= \E\left(\left(\eta-\E(\eta)\right)^4\right).
 $$
It is easily seen that the first Appell polynomials can be written as
\begin{equation}
\label{app_cum2}
Q^{(\eta)}_0(x)=1,\quad Q_1^{(\eta)}(x)=x-\kappa_1,\quad Q_2^{(\eta)}(x)=(x-\kappa_1)^2-\kappa_2. 
\end{equation} 
It is also possible to connect the cumulants and the origo moments via the Appell polynomials. 
To do this notice that
\begin{equation}
\label{cum_ori}
K'(u)=\frac{G'(u)}{G(u)},
\end{equation}
where
$$
G(u):= \E\left(\e^{u\eta}\right)=\sum_{k=0}^\infty \mu_k\,\frac{u^k}{k!}
$$
is the moment generating function of $\eta$ and $\mu_k$ are the origo moments. Equation (\ref{cum_ori}) takes the form
\[
\sum_{k=0}^\infty \kappa_{k+1}\,\frac{u^k}{k!}
=
\sum_{i=0}^\infty \mu_{i+1}\,\frac{u^i}{i!}
\sum_{k=0}^\infty Q^{(\eta)}_k(0)\,\frac{u^k}{k!}
 \]
leading to 
\begin{equation}
\label{cum_ori2}
\kappa_{m+1}=\sum_{k=0}^m{m\choose k}\,\mu_{k+1}\,Q^{(\eta)}_{m-k}(0).
\end{equation}

Notice also that if $\eta_1$ and $\eta_2$ are independent random variables with some exponential moments and if $\eta:=\eta_1 +\eta_2$ then
 \begin{equation}
\label{appell2}
Q^{(\eta)}_m(x+y)=\sum_{k=0}^m{m\choose k}Q^{(\eta_1)}_k(x)Q^{(\eta_2)}_{m-k}(y).
\end{equation}
This results after a straightforward computation from the identity
\[ 
\frac{\e^{u(x+y)}}{\E\left(\e^{u\eta}\right)}=\frac{\e^{ux)}}{\E\left(\e^{u\eta_1}\right)}\frac{\e^{uy}}{\E\left(\e^{u\eta_2}\right)}.
 \]
\begin{remark} It is perhaps amusing to notice that (\ref{appell2}) when combined with (\ref{appell1}) leads to Newton's binomial formula. Indeed,
 \begin{equation}
\label{appell21}
\nonumber
Q^{(\eta)}_m(\eta_1+\eta_2+x+y)=\sum_{k=0}^m{m\choose k}Q^{(\eta_1)}_k(\eta_1+x)
Q^{(\eta_2)}_{m-k}(\eta_2+y),
\end{equation} 
and taking expectations yield
 \begin{equation}
\label{appell22}
\nonumber
(x+y)^m=\sum_{k=0}^m{m\choose k}\,x^k\,y^{m-k}.
\end{equation}
\end{remark}

\begin{example}
Let $\eta$ be normally distributed with mean 0 and variance 1. Then
 \begin{eqnarray*}
&&
\frac{\e^{ux}}{\E\left(\e^{u\eta}\right)}=\exp\left(ux-\frac{1}{2}u^2\right)
\\ 
&&
\hskip1.4cm
=\sum_{k=0}^\infty \frac{(ux)^k}{k!}\sum_{l=0}^\infty \frac{(-u^2/2)^l}{l!}
\\ 
&&
\hskip1.4cm
=\sum_{k=0}^\infty\sum_{l=0}^\infty \frac{(-1)^l x^l}{k!\,l!\,2^l}u^{2l+k}
\\ 
&&
\hskip1.4cm
=\sum_{m=0}^\infty\frac{u^m}{m!}\ m!\sum_{l=0}^{\left[\frac{m}{2}\right]}
 \frac{(-1)^l x^{m-2l}}{(m-2l)!\,l!\,2^l}.
 \end{eqnarray*}
Consequently, 
$$
Q^{(\eta)}_n(x)=m!\sum_{l=0}^{\left[\frac{m}{2}\right]}
 \frac{(-1)^l x^{m-2l}}{(m-2l)!\,l!\,2^l},
 $$
and, hence, The Appell polynomials associated with a ${\rm N}(0,1)$-distributed random variable are the Hermite polynomials $He_n,\ n=0,1,\dots,$ (see Schoutens \cite{schoutens00} p. 52, and Abramowitz and Stegun \cite{abramowitzstegun70}  22.3.11 p. 775, 22.11 p. 785).
\end{example}
\begin{example}
\label{ex2}
We calculate the Appell polynomials of an exponentially distributed random variable. To pave the way to the applications below, consider a standard Brownian motion $B=\{B_t\,:\, t\geq 0\}$ starting from 0 and its running maximum  $M_t:=\sup_{s\leq t} B_s$. Let $T$ be an exponentially with parameter $r>0,$ distributed random variable independent of $B$. Then $M_T$ is exponentially distributed with mean $1/\sqrt{2r}.$ We find the Appell polynomials associated with $M_T$. Since
$$  
\E\left(\e^{uM_T}\right)=\frac{\sqrt{2r}}{\sqrt{2r}-u},\qquad u<\sqrt{2r}.
$$  
we have 
\begin{eqnarray*}
&&
\frac{\e^{ux}}{\E\left(\e^{uM_T}\right)}=\frac{\sqrt{2r}-u}{\sqrt{2r}}\,\e^{ux}
\\ 
&&
\hskip1.8cm
=\sum_{k=0}^\infty \frac{u^k}{k!} x^k- \sum_{k=0}^\infty \frac{u^{k+1}}{k!} \frac{x^k}{\sqrt{2r}}
\\ 
&&
\hskip1.8cm
=1+\sum_{k=1}^\infty \frac{u^k}{k!}\left(x-\frac{k}{\sqrt{2r}}\right)x^{k-1}.
 \end{eqnarray*}
Hence, the Appell polynomials are 
$$ 
Q_n^{(M_T)}(x)=\left(x-\frac{n}{\sqrt{2r}}\right)x^{n-1},\qquad n=0,1,\dots.  
 $$
Recall also that $B_T$ is Laplace-distributed with parameter $\sqrt{2r}$ and, hence, it holds
$$ 
\E\left(\e^{uB_T}\right)=\frac{2r}{{2r}-u^2},\qquad \vert u\vert <\sqrt{2r}.
 $$
Proceeding as above we find the Appell polynomial associated with $B_T:$
\begin{eqnarray*}
&&
\frac{\e^{ux}}{\E\left(\e^{uB_T}\right)}=\frac{{2r}-u^2}{{2r}}\,\e^{ux}
\\ 
&&
\hskip1.8cm
=\sum_{k=0}^\infty \frac{u^k}{k!}\left(x^2-\frac{k(k-1)}{{2r}}\right)x^{k-2}.
 \end{eqnarray*}
Let $B'$ be an independent copy of $B,$ and introduce $I'_t:=\inf_{s\leq t}B'_s.$ 
Since $\displaystyle{M_T\mathrel{\mathop=^{\rm (d)}}-I'_T } $ it holds
$$ 
Q_n^{(I_T)}(x)=\left(x+\frac{n}{\sqrt{2r}}\right)x^{n-1},\qquad n=0,1,\dots.  
 $$
It is well known (the Wiener-Hopf factorization)  that 
$$
X_T\mathrel{\mathop=^{\rm (d)}}M_T+I'_T,
  $$
and it is straightforward to check (cf. (\ref{appell2}))
 \begin{equation}
\nonumber
Q^{(X_T)}_m(x+y)=\sum_{k=0}^m{m\choose k}Q^{(M_T)}_k(x)Q^{(I_T)}_{m-k}(y).
\end{equation}
\end{example}

\subsection{L\'evy processes}
\label{sec12}
Let $X$ be a L\'evy process as introduced above and $T$ an exponential (parameter $r>0$) random variable  independent of $X.$ 
Our basic (rather restrictive) assumption on $X$ is that $X_T$ has some exponential moments, i.e., for some $\lambda >0$ 
\begin{equation}
\label{b1}
\E\left( \exp\left( \lambda\vert X_T\vert \right) \right) <\infty.
\end{equation}
 Define
 \begin{equation}
M_T=\sup_{0\leq  t<T}X_t\qquad\mbox{and}
\qquad  
I_T=\inf_{0\leq t<T}X_t. 
\label{max-min}
\end{equation}
By the Wiener-Hopf-Rogozin factorization
 \begin{equation}
 \label{wh}
X_T\mathrel{\mathop=^{\rm (d)}}M_T+I'_T,
 \end{equation}
where $I'_T$ is an independent copy of $I_T$ and ${\displaystyle{\mathop=^{\rm (d)}} }$  means that the variable on the left hand side is identical in law with the variable on the right hand side. From assumption (\ref{b1}) it follows using (\ref{wh}) that also $M_T$ and $I_T$ have some exponential moments. Moreover, under assumption (\ref{b1}), we have
\[ 
\E\left( \exp\left( \lambda X_T \right) \right)=\frac{1}{r-\psi(\lambda)}
 \]
 where 
 \begin{equation}
\psi(\gamma)=a\gamma+\frac 12b^2\gamma^2+
\int_{\R}\left( {\rm e}^{\gamma x}-1-\gamma x{\bf 1}_{\{|x|\leq 1\}} \right) \Pi(dx).
\label{eq:char-exp}
\end{equation}
is the  L\'evy-Khinchine exponent of $X$, i.e.,     
\[ 
\E\left( \exp\left( \lambda X_t \right) \right)=\exp\left( t\psi(\lambda)\right). 
 \] 
 Notation $\Pi$ stands for the L\'evy measure, i.e.,  a non negative measure defined on ${\R}\setminus\{0\}$
such that 
$\int (1\wedge x^2)\Pi(dx)<+\infty$. 
Clearly, it follows from (\ref{b1}) that for all $a>0$ and $\gamma>0$
\[ 
\int_a^\infty  {\rm e}^{\gamma x}\,\Pi(dx)+
\int_{-\infty}^{-a}  {\rm e}^{-\gamma x}\,\Pi(dx)<\infty,
 \] 
 but it is still possible to have 
 \[ 
\Pi\{(0,a) +\Pi\{(-a,0)=\infty 
 \] 
 or
 \[ 
\int_0^a x\,\Pi(dx)+\int_{-a}^0  \vert x\vert\,\Pi(dx) =\infty.
  \] 
  
In case $X$ is spectrally negative, i.e., $\Pi((0,+\infty))=0$, the process moves continuously to the right (or upwards) and $M_T$ is exponentially distributed under $\P_0$ with mean $1/\Phi(r)$ where $\Phi(r)$ is the unique positive root of the equation $\psi(\lambda)=r$ (see Bertoin \cite{bertoin96} p. 190 or Kyprianou \cite{kyprianou06} p. 213). From Example \ref{ex2} it is seen that the Appell polynomials associated with $M_T$ are 
\begin{equation}
\label{sn1}
Q_n^{(M_T)}(x)=\left(x-\frac{n}{\Phi(r)}
\right)x^{n-1},\qquad n=0,1,\dots.
\end{equation}

In case $X$ is spectrally positive, i.e., $\Pi((-\infty,0))=0$, the process moves continuously to the left (or downwards) and $-I_T$ is exponentially distributed  under $\P_0$ with mean $1/\hat\Phi(r)$ where $\hat\Phi(r)$ is the unique positive root of the equation $\psi(-\lambda)=r.$  From Example \ref{ex2} it is seen that the Appell polynomials associated with $I_T$ are 
\begin{equation}
\label{sp1}
Q_n^{(I_T)}(x)=\left(x+\frac{n}{\hat\Phi(r)}
\right)x^{n-1},\qquad n=0,1,\dots.
\end{equation}
  
\section{Optimal stopping problem}
\label{sec3}
\subsection{Review of three theorems}
\label{sec31}
We recall now the solution of the optimal stopping problem (\ref{osp}) from \cite{novikovshiryaev04} and \cite{kyprianousurya05}. 
We assume that $r>0,$ and recall that $T$ is an exponentially (with parameter $r$) distributed random variable independent of $X$.  
\begin{thm} 
\label{ns}
Assume
$$
\int_{(1,+\infty)} x^n\,\Pi(dx)<\infty. 
$$
Then $\E(M^n_T)<\infty$ and 
\begin{equation}
\nonumber
V(x):=\sup_{\tau\in {\cM}}\E_x\left({\rm e}^{-r\tau}\left(X^+_\tau\right)^n\right)
=\E_0\left(Q^{(M)}_n(M_T+x){\bf 1}_{\{M_T+x>x^*_n\}}\right)
\end{equation}
and
\begin{equation}
\nonumber
\tau^*_n=\inf\{t\geq 0\, :\, X_t>x^*_n\}
\end{equation}
is an optimal stopping time, where $Q^{(M)}$ is the Appell polynomial associated with $M_T$ and  $x^*_n$ is its largest non-negative root. 
\end{thm}

\begin{remark} 
 The proof of Theorem \ref{ns} uses the fact that 
 \begin{equation}
 \label{app_neg}
Q^{(M)}_n(x)< 0\quad{\rm for\ all}\quad x\in(0,x^*_n)
\end{equation} 
(see \cite{novikovshiryaev04}).  
This property is based on the fluctuation identity
\begin{equation}
\label{fluc}
\E_x\left({\rm e}^{-rH_a}\,X^n_{H_a}\,{\bf 1}_{\{H_a<\infty\}}\right)
=\E_0\left(Q^{(M)}_n(M_T+x){\bf 1}_{\{M_T+x>a\}}\right),
\end{equation} 
where $a>0$ and $H_a:=\inf\{t\,:\,X_t\geq a\}.$ Notice that (\ref{app_neg}) is not valid in general for Appell polynomials of a non-negative random variable. Indeed, from (\ref{app_cum2})  
we have, e.g., 
\[ 
Q^{(\eta)}_2(0)>0\quad {\rm if}\quad 2(\E(\eta))^2>\E(\eta^2).   
 \]
\end{remark} 
\begin{example}
\label{ex_sn}
Using the result in Theorem \ref{ns} it is easy to find the explicit solution of the problem for spectrally negative L\'evy processes. Indeed, recall from previous section that $M_T$ is exponentially distributed with parameter $\Phi(r),$  the Appell polynomials associated with $M_T$ are given in  (\ref{sn1}), and $x^*_n:=n/\Phi(r).$ With this data 
we obtain  
 \begin{eqnarray*}
&&
\sup_{\tau}\E_x\left(\e^{-r \tau}\left(X^+_\tau\right)^n \right)
=\E_x\left( Q^{(M_T)}_n\left(M_T \right)\,;\, M_T\geq x^*_n \right)
\\
&&
\hskip2cm
=\E_x\left(\left(M_T-\frac{n}{\Phi(r)} \right)M_T^{n-1}\,;\, M_T\geq \frac{n}{\Phi(r)}   \right) 
\\
&&
\hskip2cm
=\int_{n/\Phi(r)}^\infty\left(y-\frac{n}{\Phi(r)} \right)\,y^{n-1}\,\,\Phi(r) \e^{-\Phi(r)(y-x)}\, dy
\\
&&
\hskip2cm
=  \e^{x\Phi(r)}\left (\int_{n/\Phi(r)}^\infty y^n\,\Phi(r)\, 
\e^{-y\Phi(r)}\, dy - 
\int_{n/\Phi(r)}^\infty n\,y^{n-1}\, \e^{-y\Phi(r)}\, dy
\right )
\\
&&
\hskip2cm
= \left(\frac{n}{\Phi(r)} \right)^n\, \e^{-n}\, \e^{x\Phi(r)}.
\end{eqnarray*} 
Therefore, the value function of the optimal stopping problem is given by
\begin{equation}
\nonumber
V(x)=\begin{cases}
x^n,
 &\text{if $x\geq n/\Phi(r)$},\\
{}&\text{}\\
 \left(\frac{n}{\Phi(r)} \right)^n\, \e^{-n}\, \e^{x\Phi(r)},
&\text{if $x<n/\Phi(r)$}.\\
\end{cases}
\end{equation}
Clearly, the solution has the  smooth-fit property.
\end{example}

Next we recall the result from \cite{mordeckisalminen07} which characterizes the value function via its representing measure. The theorem is proved in \cite{mordeckisalminen07} for Hunt processes satisfying a set of assumptions. In particular, it is assumed that there exists a duality measure $m$ such that the resolvent kernel has a regular density with respect to $m$  (for the assumptions and a discussion of their validity for L\'evy processes, see \cite{mordeckisalminen07}). This density is denoted by $G_r,$ and it holds 
\[ 
\P_x(X_T\in dy)=r\,G_r(x,y)m(dy)=r\,G_r(0,y-x)m(dy)
 \]

\begin{thm}
\label{prop1}
Consider a L\'evy process  $\{X_t\}$ satisfying the assumptions made in
\cite{mordeckisalminen07}, a non-negative continuous reward function $g$,
and a discount rate $r> 0$ such that 
\begin{equation}
\label{r1}
\E_x(\sup_{t\geq 0}e^{-rt}g(X_t))<\infty.
\end{equation}
Assume that there exists a Radon 
measure $\sigma$ 
with support on the set $[x^*,\infty)$ 
such that the
function
\begin{equation}
\label{excessive}
V(x):=\int_{[x^*,\infty)}G_r(x,y)\sigma(dy)
\end{equation}
satisfies the following conditions:
\begin{itemize}
\item[\rm (a)]
$V$ is  continuous,
\item[\rm (b)]
$V(x)\to 0$ when $x\to-\infty$.
\item[\rm (c)]
$V(x)=g(x)\ \text{when $x\geq x^*$}$,
\item[\rm (d)]
$V(x)\geq g(x)\ \text{when $x<x^*$}$.
\end{itemize}
Let 
\begin{equation}\label{tau-star}
\tau^*=\inf\{t\geq 0\colon X_t\geq x^*\}.
\end{equation}
Then $\tau^*$ is an optimal stopping time and $V$ is the value
function of the optimal stopping problem for $\{X_t\}$ with the reward function $g,$
in other words,
\begin{equation*}
V(x)= \sup_{\tau\in{\cal M}}\E_x\left({\rm e}^{-r\tau} g(X_{\tau})\right)=\E_x\left( {\rm e}^{-r\tau^*}g(X_{\tau^*})\right), 
\quad x\in\R.
\end{equation*}
In case $\tau=+\infty$ 
$$
{\rm e}^{-r\tau}g(X_{\tau}):=
\limsup_{t\to\infty}{\rm e}^{-r t}g(X_t).
$$

\end{thm}

Next result, also from \cite{mordeckisalminen07}, gives a more explicit form for value function $V$ via  the maximum variable $M_T.$ The result is derived by exploiting the Wiener-Hopf-Rogozin factorization. For simplicity, we assume that $I_T$ has a density. 

\begin{thm}
\label{corollary:onesided}
Assume that the conditions of Theorem  \ref{prop1} hold.
Then, there exists a function $H\colon [x^*,\infty)\to\R$ such
that the value function $V$ in  \eqref{excessive} 
satisfies
\begin{equation*}
V(x)=\E_x\left(H(M_T)\,;\,M_T\geq x^*\right),
\qquad
x\leq x^*.
\end{equation*}
Moreover, the function $H$ has the explicit representation 
\begin{equation}\label{q}
H(z)=r^{-1}\int_{x^*}^{z}f_I(y-z)\sigma(dy),\quad  z\geq x^*.
\end{equation}
where $f_I$ is the density of the distribution of $I_T.$  
\end{thm}

\subsection{Representing measure}
\label{sec32}

Let now $V$ denote the value function for problem (\ref{osp}) as given in Theorem \ref{ns}:
\begin{equation}
\label{V}
V(x)=\E_0\left(Q^{(M)}_n(M_T+x){\bf 1}_{\{M_T+x>x^*_n\}}\right).
\end{equation}  
This function has the properties (a)-(c) of Theorem \ref{prop1}; (a) and (b) follow from monotone convergence, (c) and (d) from the properties of the Appell polynomials (or from the fact that $V$ is indeed the value function of the problem). Therefore, it is natural to ask whether it is possible to find a measure $\sigma_n,$ say, such that (\ref{q}) holds, i.e.,
\begin{equation}\label{qq}
Q^{(M)}_n(z)=r^{-1}\int_{x_n^*}^{z}f_I(y-z)\sigma_n(dy),\quad  z\geq x_n^*.
\end{equation} 
From this equation it is possible to derive an expression for the Laplace transform of $\sigma.$
\begin{proposition}
\label{re}
There exists a measure $\sigma_n$ such that (\ref{qq}) holds. The Laplace transform $\widehat\sigma_n$ of 
this measure is given by 
\begin{eqnarray}
&&
\nonumber
\widehat\sigma_n(\gamma):=\int_{x_n^*}^{\infty} \, \e^{-\gamma y}\, \sigma_n(dy)
\\
\label{req3}
&&\hskip1.2cm
=
r\,\widehat{q}_n(\gamma, x_n^*)/\E\left (\e^{\gamma I_T}\right )
\\
\label{req31}
&&\hskip1.2cm
=\sum_{k=0}^\infty \frac{\gamma^k}{k!}\,\widehat{q}_n(\gamma,x_n^*)\, Q^{(I)}_k(0),  
\end{eqnarray}
where
\begin{eqnarray}
\label{q1}
&&
\nonumber
\widehat{q}_n(\gamma,x_n^*):=\int_{x_n^*}^{\infty} \e^{-\gamma z}Q_n^{(M)}(z)\, dz
\\
&&\hskip1.7cm
=\sum_{i=0}^n\frac{n!}{(n-i)!}\frac{{\e^{-\gamma x_n^*}}}{\gamma^{i+1}}
\, Q^{(M)}_{n-i}(x_n^*)
\end{eqnarray}
with $Q^{(M)}_{n-i}(x_n^*)\geq 0,\ i=0,1,\dots,n.$ 
\end{proposition}
\begin{proof} Multiply equation (\ref{qq}) with $\e^{-\gamma z}$ and  integrate over $(x^*,\infty)$  to obtain (to simply the notation we omit the subindex $n$) 
\begin{eqnarray*}
\label{eq1}
&&
\nonumber
\hskip-.6cm
\int_{x^*}^{\infty} \e^{-\gamma z}Q^{(M)}(z)\, dz
=r^{-1}\int_{x^*}^{\infty} dz  \e^{-\gamma z}\int_{x^*}^{z}\sigma(dy)\,f_I(y-z)
\\
&&
\nonumber
\hskip3cm
=r^{-1}\int_{x^*}^{\infty}  \sigma(dy)\, 
\int_{y}^{\infty} dz\, \e^{-\gamma z}\,f_I(y-z)
\\
&&
\nonumber
\hskip3cm
=r^{-1}\int_{x^*}^{\infty} \sigma(dy)\, \e^{-\gamma y} 
\int_{y}^{\infty} dz\, \e^{-\gamma (z-y)}\,f_I(y-z)
\\
&&
\hskip3cm
=r^{-1}\int_{x^*}^{\infty} \sigma(dy)\, \e^{-\gamma y} 
\int_{0}^{\infty} dz\, \e^{-\gamma u}\,f_I(-u).
\end{eqnarray*}
In other words, $\widehat{\sigma}$ is well defined (by the assumption that $I_T$ has some exponential moments) and satisfies
\[ 
\widehat\sigma(\gamma)=r\,\widehat{q}(\gamma,x^*)/\E\left (\e^{\gamma I_T}\right ),
 \]
which is (\ref{req3}). To obtain (\ref{req31}) use the definition of Appell polynomials (see (\ref{appell}), and identity (\ref{q1}) is a straightforward integration using (\ref{appellx}). Finally, notice that the positive zeros $x^*_n$ of $Q^{(M)}_n$, respectively, satisfy
\[  
x^*_1\leq x^*_2\leq\dots 
\]
This follows readily again from (\ref{appellx}) and the fact that $x^*_k$ is the unique positive zero of  $Q^{(M)}_k.$ Hence, $Q^{(M)}_k(x^*_n)\geq 0$ for $k=0,1,\dots, n.$
\end{proof}

In case $X$ is spectrally negative the value function can be determined explicitly  as demonstrated in Example \ref{ex_sn}. However, the Laplace transform of the representing measure does not seem to have an explicit inversion. To discuss this case more in detail, recall that  $M_T$ is exponentially distributed with parameter $\Phi(r).$ Hence, we can relax our assumptions on exponential moments. It holds for $\gamma\geq 0$
\[ 
\E\left( \e^{\gamma X_T}\right) =\frac{r}{r-\psi(\gamma)},
 \]
and, by the Wiener-Hopf factorization,
\begin{equation}
\label{b11}
\E\left( \e^{\gamma I_T}\right) =\E\left( \e^{\gamma X_T}\right)/\E\left( \e^{\gamma M_T}\right)=\frac{r\left(\Phi(r)-\gamma \right) }{\Phi(r)\left( r-\psi(\gamma)\right) }.
\end{equation}
Notice also that $Q^{(M)}_k(x^*_n)=(n-k)n^k/(\Phi(r))^k.$ 
To conclude this discussion, we formulate the following
\begin{corollary} 
\label{sigma_sp}
If $\{X_t\}$ is spectrally negative then 
\begin{eqnarray}
&&
\nonumber
\hskip-1.3cm
\widehat\sigma_n(dx)=r\,\widehat{q}_n(\gamma, x_n^*)/\E\left (\e^{\gamma I_T}\right )
=\widehat{q}_n(\gamma, x_n^*)
\frac{\Phi(r)\left( r-\psi(\gamma)\right) }{\Phi(r)-\gamma }
\end{eqnarray} 
and 
\begin{eqnarray}
\label{q11}
&&
\nonumber
\widehat{q}_n(\gamma,x_n^*):=\int_{x_n^*}^{\infty} \e^{-\gamma z}Q_n^{(M)}(z)\, dz
\\
&&\hskip1.7cm
=\sum_{i=0}^n\frac{n!}{(n-i)!}\frac{{\e^{-\gamma x_n^*}}}{\gamma^{i+1}}
\, \frac{i}{\Phi(r)}\,\left( \frac{n}{\Phi(r)}\right)^{n-i}.
\end{eqnarray}
\end{corollary}

For spectrally positive L\'evy processes we do not, in general, have any explicit knowledge of the Appell polynomials $Q^{(M)}$ and, hence, also the value function $V$ remains hidden. However, in this case we may charaterize $\sigma_n$ nicely. This is possible since now $-I_T$ is exponentially distributed and we see from (\ref{qq}), roughly speaking, that $Q^{(M)}$ is the Laplace transform of $\sigma_n.$ We obtain also a new expression for the value function. 
\begin{corollary} 
\label{sigma_sn}
If $X$ is spectrally positive then 
\begin{equation}
\label{nice}
\sigma(dx)=r\,Q^{(X)}_n(x)dx,\quad  x\geq x^*_n,  
\end{equation}
and, hence,
\begin{equation}
\label{vnice}
V(x)=\E\left(Q^{(X)}_n(X_T+x)\,;\,X_T+x\geq x^*_n  \right). 
\end{equation}
\end{corollary}
\begin{proof}
The Appell polynomials for $I_T$ are given explicitly in (\ref{sp1}), and  we have $Q_0^{(I_T)}(0)=1,$ $Q_1^{(I_T)}(0)=1/\hat\Phi(r),$ and $Q_n^{(I_T)}(0)=0$ for $n=2,3,\dots$. From (\ref{req31}) we obtain
\begin{equation}
\label{eq5}
\widehat\sigma(\gamma)=r\,
 \widehat{q}(\gamma)+ \frac{r}{\hat\Phi(r)}\,\gamma\,\widehat{q}(\gamma).
\end{equation}
Consider
\begin{eqnarray*}
&&
\gamma\,\widehat{q}(\gamma)=
\int_{x^*}^{\infty}\gamma\, \e^{-\gamma z}Q_n^{(M)}(z)\, dz
\\
&&
\hskip1.2cm
=
\int_{x^*}^{\infty}\gamma\, \e^{-\gamma z}\left (\int_{x^*}^{\infty} \frac{d}{dx}Q_n^{(M)}(x)\,dx\right )\, dz
\\
&&
\hskip1.2cm
=
\int_{x^*}^{\infty} \frac{d}{dx}Q_n^{(M)}(x)\, \e^{-\gamma x}\, dx 
\\
&&
\hskip1.2cm
=
\int_{x^*}^{\infty} n\,Q_{n-1}^{(M)}(x)\, \e^{-\gamma x}\, dx,
\end{eqnarray*}
where we have changed the order of integration, used the fact $x_n^*$ is a zero of $Q_n^{(M)},$ and applied the differentiation formula (\ref{appellx}). Consequently, we have done the inversion and it holds
 \begin{eqnarray*}
&& 
\sigma(z)=r Q^{(M)}_n(z)+ n \frac{r}{\hat\Phi(r)}Q^{(M)}_{n-1}(z) 
\\ 
&&
\hskip1cm
=r \left( Q^{(M)}_n(z)Q^{(I)}_0(0)+ n Q^{(M)}_{n-1}(z)Q^{(I)}_1(0)\right)  
\\ 
&&
\hskip1cm
=r \sum_{k=0}^{n} {n\choose k}Q^{(M)}_{n-k}(z)Q^{(I)}_{k}(0)
\\ 
&&
\hskip1cm
=r\,Q^{(X)}_{n}(z), 
 \end{eqnarray*}
since $Q^{(I)}_{k}(0)=0$ for $k=2,3,\dots$ and, in the last step, formula (\ref{appell2}) is applied.  
\end{proof}

We conclude by studying Brownian motion from the point of view of Corollary 
\ref{sigma_sn}.
\begin{example} The resolvent kernel with respect to the Lebesgue measure is given by
(see, e.g., Borodin and Salminen \cite{borodinsalminen02} p. 120)
\[ 
G_r(x,y)=\frac{1}{\sqrt{2 r}}\ \e^{-\sqrt{2 r}\vert x-y\vert},
 \]
and, from Example \ref{ex2}, $x^*_n=n/\sqrt{2r}$ and  
\[ 
Q^{(X_T)}_{n}(x)=\left( x^2-\frac{n(n-1)}{2r}\right) x^{n-2}.
 \] 
It is straightforward to check that $V$ as given in Example \ref{ex_sn} satisfies 
 \begin{eqnarray*}
&& 
V(x)=r\int_{x^*_n}^\infty G_r(x,y)\,Q_n^{(X_T)}(y)dy.
 \\
 &&
 \hskip1cm
 =\E_x\left(Q^{(X_T)}_{n}(X_T)\,;\,X_T\geq x^*_n\right). 
 \end{eqnarray*}

\end{example}

\bibliographystyle{plain}
\bibliography{pape1}

\end{document}